\documentclass[11pt]{article}
\usepackage{amsmath,amsthm,amsfonts,amssymb,bm,wasysym}
\usepackage{subfigure}
\usepackage{graphicx}
\usepackage[usenames]{color}
\usepackage{verbatim}
\usepackage{epsfig}

\parskip \medskipamount
\newtheorem{theorem}{Theorem}[section]
\newtheorem{definition}[theorem]{Definition}

\numberwithin{equation}{section}
\newtheorem{lemma}[theorem]{Lemma}
\newtheorem{proposition}[theorem]{Proposition}
\newtheorem{corollary}[theorem]{Corollary}

\newtheorem{remark}[theorem]{Remark}
\newtheorem{example}[theorem]{Example}
\newtheorem{claim}[theorem]{Claim}

\numberwithin{equation}{section}

\def\N{\mathbb{N}}
\def\Z{\mathbb{Z}}

\def\R{\mathbb{R}}
\def\C{\mathbb{C}}

\def\T{\mathcal{T}}

\def\F{\mathcal{F}}

\def\B{\mathcal{B}}

\renewcommand{\phi}{\varphi}
\renewcommand{\epsilon}{\varepsilon}

\allowdisplaybreaks

\newcommand{\1}{{\text{\Large $\mathfrak 1$}}}

\newcommand{\vol}{\mathrm{vol}}
\renewcommand{\emptyset}{\varnothing}

\def\C{{\mathcal C}}

\newcommand{\til}{\widetilde}

\newcommand{\pr}[1]{\mathbb{P}\!\left(#1\right)}
\newcommand{\E}[1]{\mathbb{E}\!\left[#1\right]}

\newcommand{\prcond}[3]{\mathbb{P}_{#3}\!\left(#1\;\middle\vert\;#2\right)}
\newcommand{\econd}[2]{\mathbb{E}\!\left[#1\;\middle\vert\;#2\right]}

\newcommand{\tn}{|\kern-.1em|\kern-0.1em|}

\newcommand{\lowdim}{\underline{\dim}_M}
\newcommand{\updim}{\overline{\dim}_M}

\newcommand\be{\begin{equation}}
\newcommand\ee{\end{equation}}

\begin{document}

\title{Minkowski dimension of Brownian motion with drift}

\author{
Philippe H.\ A.\ Charmoy\thanks{University of Oxford, Oxford, UK; charmoy@maths.ox.ac.uk} \and Yuval Peres\thanks{Microsoft Research, Redmond, Washington, USA; peres@microsoft.com} \and Perla Sousi\thanks{University of Cambridge, Cambridge, UK;   p.sousi@statslab.cam.ac.uk}
}
\maketitle
\begin{abstract}

We study fractal properties of the image and the graph of Brownian motion in $\R^d$ with an arbitrary c{\`a}dl{\`a}g drift $f$. We prove that the Minkowski (box) dimension of both the image and the graph of $B+f$ over $A\subseteq [0,1]$ are a.s.\ constants. We then show that for all $d\geq 1$ the Minkowski dimension of $(B+f)(A)$ is at least the maximum of the Minkowski dimension of $f(A)$ and that of $B(A)$.
We also prove analogous results for the graph. For linear Brownian motion, if the drift $f$ is continuous and $A=[0,1]$, then the corresponding inequality for the graph is actually an equality.
\newline
MSC 2010 \emph{subject classifications.}
Primary: 60J65; Secondary
60G17, 28A12.  
\end{abstract}

\section{Introduction}

Let $(B_t)$ be a standard Brownian motion in $\R^d$ and $f:[0,1]\to \R^d$ a c{\`a}dl{\`a}g function. By the Cameron-Martin theorem, the law of $B+f$ is equivalent to the law of $B$ when $f$ is in the Dirichlet space
$$
D[0,1] = \left\{ f \in C[0,1]: f(t) = \int_0^t g(s) ds \text{ for some function } g \in L^2[0,1] \right\},
$$
and singular to the law of $B$ otherwise.

In~\cite{PS10} it is shown that if $f$ is any continuous function, then the Hausdorff dimension of the image and the graph of $B+f$ are almost surely constants. In the same paper it is also proved that if $A$ is a closed subset of $[0,1]$, then the Hausdorff dimension of $(B+f)(A)$ is at least $\max\{\dim_H B(A), \dim_H f(A)\}$ and similarly for the dimension of the graph of $B+f$ over $A$.

In this paper we prove analogous results for the Minkowski (or otherwise called box) dimension of the same sets. We would like to emphasize that the presence of the drift $f$ implies that we cannot use techniques relying on self-similarity of the paths.

Before stating our main results, we recall the definition of Minkowski dimension. For other equivalent definitions and properties see~\cite[Definition~3.1]{Falconer}.

\begin{definition}\rm{
Let $A$ be a non-empty bounded subset of $\R^d$. For $\epsilon>0$ let $P(A,\epsilon)$ be the maximum number of disjoint balls of radius $\epsilon$ with centers in $A$:
\[
P(A,\epsilon) = \max \left\{k: \exists \ x_1,\ldots, x_k \in A \text{ s.t.\ } \B(x_i,\epsilon)\cap \B(x_j,\epsilon) =\emptyset \text{ if } i\neq j\right\}.
\]
The \emph{upper} and \emph{lower Minkowski dimensions} of $A$ are defined as
\[
\updim(A) = \limsup_{\epsilon \to 0} \frac{\log P(A,\epsilon)}{\log\epsilon^{-1}} \ \text{ and } \ \lowdim(A) = \liminf_{\epsilon \to 0} \frac{\log P(A,\epsilon)}{\log\epsilon^{-1}}
\]
respectively. Whenever these two limits are equal, we call the common value the Minkowski dimension of $A$.
}
\end{definition}

Let $f:[0,1] \to \R^d$ be a c{\`a}dl{\`a}g function and $A$ a subset of $[0,1]$.
In this paper we first prove that the Minkowski dimension of the image and the graph of $B+f$ over the set $A$ are a.s.\ constants. The 0-1 law
(Theorem~2.1) from \cite{PS10} used to prove the a.s.\ constancy of the Hausdorff dimension of $(B+f)(A)$
cannot be used to prove the a.s.\ constancy in this case, since the
Minkowski dimension does not satisfy the countable stability property; this means that the Minkowski dimension of a countable union of sets is not in general the supremum of their dimensions.

\begin{theorem}\label{thm:constancydimimage}
	Let $(B_t)$ be a standard Brownian motion in $d$ dimensions. Let $f:[0,1] \to \R^d$ be a c{\`a}dl{\`a}g function and let $A$ be a subset of $[0,1]$. Then, there exist constants $c_1$ and $c_2$ such that, almost surely,
	$$
	\underline{\dim}_M (B+f)(A) = c_1 \ \text{ and } \ \overline{\dim}_M (B+f)(A) = c_2.
	$$
\end{theorem}

For a function $h:[0,1] \to \R^d$ and a set $A\subseteq [0,1]$ we denote by $G_A(h) = \{ (t,h(t)): t\in A\}$ the graph of $h$ over $A$.

\begin{theorem}\label{thm:constancydimgraph}
	Let $(B_t)$ be a standard Brownian motion in $d\geq 1$ dimensions, $f:[0,1] \to \R^d$ a c{\`a}dl{\`a}g function and $A$ a subset of $[0,1]$. Then, there exist constants $c_3$ and $c_4$ such that, almost surely,
	$$
	\underline{\dim}_M G_A(B+f) = c_3 \ \text{ and } \ \overline{\dim}_M G_A(B+f) = c_4.
	$$
\end{theorem}

We prove Theorems~\ref{thm:constancydimimage} and~\ref{thm:constancydimgraph} in Section~\ref{sec:0-1law} by relating the Minkowski dimension to the expected volume of the ``sausage" around the graph or the image. In the same section we also give an alternative proof of Theorem~\ref{thm:constancydimgraph} using L{\'e}vy's construction of Brownian motion.

Having established that the Minkowski dimension of the image and the graph of Brownian motion with a c{\`a}dl{\`a}g drift are a.s.\ constants we show that adding a deterministic drift to the Brownian motion cannot decrease the dimension of the image and the graph.

\begin{theorem}\label{thm:minkba}
Let $(B_t)$ be a standard Brownian motion in $d\geq 1$ dimensions. Let $A$ be a subset of $[0,1]$ and $f:[0,1]\to \R^d$ a c{\`a}dl{\`a}g function. Then almost surely
\begin{align*}
\updim (B+f)(A) & \geq \max\{\updim B(A), \updim f(A) \} \\
\lowdim (B+f)(A) & \geq \max\{\lowdim B(A), \lowdim f(A)\}.
\end{align*}
\end{theorem}

McKean's theorem (see for instance~\cite[Theorem~4.33]{BM}) states that if $A$ is a closed subset of $[0,\infty)$, then $\dim_H B(A) = (2\dim_H A) \wedge d$, where $\dim_H$ stands for the Hausdorff dimension. In the case of Minkowski dimension there cannot be such a formula as the following corollary shows.

\begin{corollary}\label{thm:dimimagelinearBM}
	Let $(B_t)$ be a standard Brownian motion in $d\geq 1$ dimensions and $f:[0,1]\to \R^d$ a c{\`a}dl{\`a}g function. Then, for every subset $A$ of $[0,1]$, if $d=1$, then almost surely,
	$$
	\underline{\dim}_M (B+f)(A) \geq \frac{2 \underline{\dim}_M A}{\underline{\dim}_M A +1},\quad \text{and} \quad \overline{\dim}_M (B+f)(A) \geq \frac{2 \overline{\dim}_M A}{\overline{\dim}_M A +1}.
	$$
The lower bounds can be achieved. If $d\geq 2$, then the right hand side in these inequalities is replaced by $2\lowdim A$ and $2\updim A$ respectively.
\end{corollary}

\begin{remark}\rm{
Inequalities analogous to Corollary~\ref{thm:dimimagelinearBM} for packing dimension of images $X(A)$ where $X$ is a multi-parameter fractional Brownian motion were established by Talagrand and Xiao in~\cite{TalagrandXiao}.
}
\end{remark}

We prove Theorem~\ref{thm:minkba} and Corollary~\ref{thm:dimimagelinearBM} in Section~\ref{sec:image}.
\newline
We now state our results concerning the Minkowski dimension of the graph of $B+f$. We prove them in Section~\ref{sec:graph}.

\begin{theorem}\label{thm:dimgraphBMbrownian}
	Let $(B_t)$ be a standard Brownian motion in $d$ dimensions and let $f:[0,1] \to \R^d$ be a c{\`a}dl{\`a}g function. Then, for every subset $A$ of $[0,1]$, we have, almost surely,
	\begin{align*}
		\updim G_A(B+f) & \geq \max\{\updim G_A(B), \updim G_A(f)\} \\
		\lowdim G_A(B+f) & \geq \max\{\lowdim G_A(B), \lowdim G_A(f)\}.
	\end{align*}
\end{theorem}

In one dimension, when the drift function $f$ is continuous and $A=[0,1]$, equality is achieved in the inequalities of Theorem~\ref{thm:dimgraphBMbrownian}.

\begin{theorem}\label{thm:graphlinearBM}
	Let $(B_t)$ be a standard Brownian motion in one dimension and $f:[0,1] \to \R$ a continuous function. Then, almost surely,
	\begin{align*}
		\underline{\dim}_M G_{[0,1]} (B+f) & = \max\{\lowdim G_{[0,1]}(B), \underline{\dim}_M G_{[0,1]} (f)\},\\
		\overline{\dim}_M G_{[0,1]} (B+f) & = \max\{\lowdim G_{[0,1]}(B), \overline{\dim}_M G_{[0,1]} (f)\}.
	\end{align*}
\end{theorem}

\begin{remark}
\rm{
The equalities in Theorem~\ref{thm:graphlinearBM} can fail if $f$ is not continuous.
In Section~\ref{sec:example} we describe a c{\`a}dl{\`a}g function $f$ such that $\updim G_{[0,1]}(f) = 5/3$ and $\updim G_{[0,1]}(B+f) \geq 7/4$ a.s.
}
\end{remark}

\begin{remark}\rm{
We note that all the results stated above readily extend to packing dimension due to its representation in terms of upper Minkowski dimension, see~\cite[Proposition~3.8]{Falconer}.
}
\end{remark}

{\bf Related results.} Fractal properties of images $X(A)$, where $X$ is a L{\'e}vy process or a multi-parameter fractional Brownian motion were investigated in~\cite{KhoshSchilXiao, KhoshXiao, PerkinsTaylor, TalagrandXiao, Xiao}. Here we restrict attention to Brownian motion; the new feature is the effect of the drift function $f$.

\section{A 0-1 law}\label{sec:0-1law}

In this section we prove Theorems~\ref{thm:constancydimimage} and~\ref{thm:constancydimgraph} by first stating and proving a more general result for any c{\`a}dl{\`a}g adapted process with stationary and independent increments. At the end of the section we give a second proof of Theorem~\ref{thm:constancydimgraph} using L{\'e}vy's construction of Brownian motion.

We introduce some notation that will be used throughout the paper. If $g:\R_+\to \R^d$ is a measurable function and $A$ a subset of $[0,1]$, then for any $r>0$ we define
\begin{align}\label{eq:defvolume}
V_g(A,r) = \vol\left( \cup_{s\in A} \B(g(s),r) \right),
\end{align}
where $\B(x,r)$ stands for the ball centered at $x$ of radius $r$.

We are now ready to state the main result of this section.

\begin{proposition}\label{pro:general}
Let $(\F_t)$ be a right continuous filtration and $(X_t)$ a c{\`a}dl{\`a}g adapted process taking values in $\R^d$, $d\geq 1$, with stationary and independent increments.
Let $f:[0,1] \to \R^d$ be a c{\`a}dl{\`a}g function and $A$ a subset of $[0,1]$. Then almost surely we have
\begin{align*}
&\updim (X+f)(A) = \limsup_{\epsilon \to 0} \frac{\log \E{V_{X+f}(A,\epsilon)}}{\log \frac{1}{\epsilon}}
\\
& \lowdim (X+f)(A)= \liminf_{\epsilon \to 0} \frac{\log \E{V_{X+f}(A,\epsilon)}}{\log \frac{1}{\epsilon}}.
\end{align*}
\end{proposition}

Before proving it we explain how Theorems~\ref{thm:constancydimimage} and~\ref{thm:constancydimgraph} follow.

\begin{proof}[{\bf Proof of Theorem~\ref{thm:constancydimimage}}]
Setting $X_t = B_t$, since Brownian motion satisfies the assumptions of Proposition~\ref{pro:general}, the theorem follows.
\end{proof}

\begin{proof}[{\bf Proof of Theorem~\ref{thm:constancydimgraph}}]

For $t\in \R_+$, let $X_t = (t,B_t)$ and $g(t) = (t,f(t))$. Then $X$ and $g$ clearly satisfy the assumptions of Proposition~\ref{pro:general}, and hence this finishes the proof.
\end{proof}

We now devote the rest of the section to the proof of Proposition~\ref{pro:general}. First we state a standard fact about Minkowski dimensions which can be found e.g.\ in \cite[Proposition~3.2]{Falconer}.

\begin{claim}\label{cl:minkvolgeneral}
Let $A$ be a bounded subset of $\R^d$. Then
\begin{align*}
&\updim A = \limsup_{\epsilon \to 0} \frac{\log \vol(A+\B(0,\epsilon))}{\log \frac{1}{\epsilon}} + d \\
& \lowdim A = \liminf_{\epsilon \to 0} \frac{\log \vol(A+\B(0,\epsilon))}{\log \frac{1}{\epsilon}} + d.
\end{align*}
\end{claim}

\begin{proof}[{\bf Proof}]
Let $P(\epsilon)$ be the maximum number of disjoint $\epsilon$-balls with centers $x_1,\ldots, x_{P(\epsilon)}$ in $A$. It is easy to see that
\[
\bigcup_{i=1}^{P(\epsilon)} \B(x_i,\epsilon) \subseteq A+ \B(0,\epsilon) \subseteq \bigcup_{i=1}^{P(\epsilon)} \B(x_i,2\epsilon).
\]
It then follows that
\[
 c(d) \epsilon^d P(\epsilon) \leq \vol(A+\B(0,\epsilon)) \leq  c(d) (2\epsilon)^d P(\epsilon).
\]
Taking logarithms of both sides, dividing by $\log \epsilon^{-1}$ and taking the limit as $\epsilon$ goes to $0$ completes the proof of the claim.
\end{proof}

The main ingredient of the proof of Proposition~\ref{pro:general} is the following lemma on the concentration of the volume of the sausage around its mean.

\begin{lemma}\label{lem:volconc}
Let $(\F_t)$ be a right continuous filtration and $(X_t)$ a c{\`a}dl{\`a}g adapted process taking values in $\R^d$, $d\geq 1$, with stationary and independent increments.
Let $f:[0,1] \to \R^d$ be a c{\`a}dl{\`a}g function and $A$ a subset of $[0,1]$. Then almost surely we have
\begin{align*}
&\limsup_{\epsilon \to 0} \frac{\log V_{X+f}(A,\epsilon)}{\log \frac{1}{\epsilon}} = \limsup_{\epsilon \to 0} \frac{\log \E{V_{X+f}(A,\epsilon)}}{\log \frac{1}{\epsilon}}
\\
& \liminf_{\epsilon \to 0} \frac{\log V_{X+f}(A,\epsilon)}{\log \frac{1}{\epsilon}} = \liminf_{\epsilon \to 0} \frac{\log \E{V_{X+f}(A,\epsilon)}}{\log \frac{1}{\epsilon}}.
\end{align*}
\end{lemma}

\begin{claim}\label{cl:stoppingtime}
Let $(\F_t)$ be a right continuous filtration and $(X_t)$ a c{\`a}dl{\`a}g adapted process taking values in $\R^d$, $d\geq 1$. Let $D$ be an open set in $\R^d$ and $F$ a subset of $[0,1]$. Then
\[
\tau = \inf\{ t\in F: X_t \in D\}
\]
is a stopping time.
\end{claim}

\begin{proof}[{\bf Proof}]
Let $F_\infty$ be a countable dense subset of $F$. Then for all $t\in [0,1]$ we deduce
\[
\{\tau <t\} = \cup_{q \in F_\infty, q<t} \{X_q \in D\},
\]
since $X$ is c{\`a}dl{\`a}g and $D$ is an open set. Hence $\{\tau<t\} \in \F_t$. Writing
\[
\{\tau\leq t\} = \bigcap_{n} \{ \tau<t+1/n\},
\]
we get that $\{\tau\leq t\} \in \F_{t+} = \F_t$.
\end{proof}

\begin{proof}[{\bf Proof of Lemma~\ref{lem:volconc}}]
First notice that by the monotonicity of the volume we have
\begin{align*}
&\limsup_{\epsilon \to 0} \frac{\log V_{X+f}(A,\epsilon)}{\log \frac{1}{\epsilon}} = \limsup_{k\to \infty} \frac{\log V_{X+f}(A,2^{-k})}{\log 2^k}  \ \ \text{ and }
\\
& \liminf_{\epsilon \to 0} \frac{\log V_{X+f}(A,\epsilon)}{\log \frac{1}{\epsilon}} = \liminf_{k\to \infty} \frac{\log V_{X+f}(A,2^{-k})}{\log 2^k}.
\end{align*}
Hence it suffices to show that a.s.\ for all large enough $k$ we have
\[
\frac{1}{2k}\E{V_{X+f}(A,2^{-k})} \leq V_{X+f}(A,2^{-k}) \leq k^2 \E{V_{X+f}(A,2^{-k})}.
\]
The upper bound follows easily from Markov's inequality and the Borel Cantelli lemma. We now show the lower bound.
In fact, note that by Borel Cantelli, it suffices to show that for all $k$ we have
\begin{align}\label{eq:upperboundborel}
\pr{V_{X+f}(A,2^{-k}) \geq \frac{1}{2k} \E{V_{X+f}(A,2^{-k})}} \geq 1 - \left(\frac{7}{8}\right)^k.
\end{align}

We first show that for any measurable $F\subseteq \R_+$ and any $\delta>0$ we have
\begin{align}\label{eq:secondmoment}
\E{(V_{X+f}(F,\delta))^2} \leq 2 (\E{V_{X+f}(F,\delta)})^2.
\end{align}
For any $x\in \R^d$ we write
\[
\tau_x = \inf\{t \in F: X(t) + f(t) \in \B(x,\delta)\}
\]
with the convention that $\tau_x$ is infinite, if $X+f$ does not hit the ball $\B(x,\delta)$. We have
\begin{align}\label{eq:volumesquared}
\E{(V_{X+f}(F,\delta))^2}  & = \int_{\R^d} \int_{\R^d} \pr{\tau_x <\infty, \tau_y <\infty}\, dx \,dy =
2 \int_{\R^d}\int_{\R^d} \pr{\tau_x \leq \tau_y<\infty} \,dx \,dy \nonumber
\\
\nonumber
&= 2\int_{\R^d} \pr{\tau_x<\infty} \int_{\R^d}\prcond{\tau_x\leq \tau_y<\infty}{\tau_x <\infty}\, dy\,dx \\
&= 2\int_{\R^d} \pr{\tau_x<\infty} \econd{V_{X+f}(F\cap [\tau_x,\infty),\delta)}{\tau_x<\infty} \,dx.
\end{align}
Since $X$ and $f$ are c{\`a}dl{\`a}g and the filtration is right continuous, it follows from Claim~\ref{cl:stoppingtime} that $\tau_x$ is a stopping time.
By the stationarity, the independence of increments and the c{\`a}dl{\`a}g property of $X$, we get that $X$ satisfies the strong Markov property (see~\cite[Proposition~I.6]{Bertoin}).
Thus the conditional law of the process $\{X(\tau_x+s)-X(\tau_x)\}_{s\geq 0}$ given that   $\{\tau_x<\infty\}$,  is identical to the law of  $\{X(s)\}_{s\geq 0}$.
Let $X'$ be a process independent of $X$ but with the same law as $X$. The Markov property of $X'$ and its independence from $\tau_x$  implies that the conditional law of the process
$\{X'(\tau_x+s)-X'(\tau_x)\}_{s\geq 0}$ given that   $\{\tau_x<\infty\}$,  is also identical to the law of  $\{X(s)\}_{s\geq 0}$. Therefore given $\tau_x<\infty$, the two random paths
 $\{X(t)-X(\tau_x)\}_{t\geq \tau_x}$ and $\{X'(t)-X'(\tau_x)\}_{t\geq \tau_x}$ have the same law. Since volume is unaffected by translation,
\begin{align*}
\econd{V_{X+f}(F\cap [\tau_x,\infty),\delta)}{\tau_x<\infty} & = \econd{V_{X'+f}(F\cap [\tau_x,\infty),\delta)}{\tau_x<\infty} \\ &\leq \E{V_{X'+f}(F,\delta)}= \E{V_{X+f}(F,\delta)},
\end{align*}
and hence this together with~\eqref{eq:volumesquared} concludes the proof of~\eqref{eq:secondmoment}.

Therefore, from~\eqref{eq:secondmoment}, applying the second moment method to the random variable $V_{X+f}(F,\delta)$ we get that for any set $F$ and any $\delta>0$
\begin{align}\label{eq:momentmethod}
\pr{V_{X+f}(F,\delta) \geq \frac{1}{2} \E{V_{X+f}(F,\delta)}} \geq \frac{1}{8}.
\end{align}
We set $t_0=0$.
It is easy to see that $\E{V_{X+f}(A\cap [0,t],2^{-k})}$ is continuous as a function of $t$. Hence for $j=1,\ldots, k$ we can define
\begin{align*}
t_j = \inf\left\{ t\geq 0: \E{V_{X+f}(A\cap [0,t],2^{-k})} = \frac{j}{k}\E{V_{X+f}(A,2^{-k})}\right\}
\end{align*}
and we write $I_j = [t_{j-1},t_j]$. By the subadditivity property of the volume, we get that for all $j$
\[
\E{V_{X+f}(A\cap I_j,2^{-k})} \geq \E{V_{X+f}(A\cap [0,t_j],2^{-k})} -\E{V_{X+f}(A\cap [0,t_{j-1}],2^{-k})}
=\frac{1}{k}\E{V_{X+f}(A,2^{-k})}.
\]
Therefore we get
\begin{align*}
&\pr{V_{X+f}(A,2^{-k})  \geq \frac{1}{2k}\E{V_{X+f}(A,2^{-k})}} \geq \pr{\exists j: \ V_{X+f}(A\cap I_j, 2^{-k}) \geq \frac{1}{2} \E{V_{X+f}(A\cap I_j,2^{-k})}} \\
& = 1 - \prod_{j=1}^{k} \pr{V_{X+f}(A\cap I_j, 2^{-k}) < \frac{1}{2} \E{V_{X+f}(A\cap I_j,2^{-k})}} \geq 1 - \left(\frac{7}{8}\right)^k,
\end{align*}
where the equality follows by the independence of the increments of $X$ and the last inequality follows from~\eqref{eq:momentmethod}. This finishes the proof of~\eqref{eq:upperboundborel}, and hence concludes the proof of the lemma.
\end{proof}

\begin{remark}\rm{
We note that if $X$ and $f$ are c{\`a}dl{\`a}g and $A$ is a subset of $[0,1]$, then $V_{X+f}(A,\epsilon)$ is a random variable. Indeed, let $A_\infty$ be a countable dense subset of $A$, then $V_{X+f}(A,\epsilon) = V_{X+f}(A_\infty,\epsilon)$.
Now $V_{X+f}(A_\infty,\epsilon) = \lim_{n\to \infty} V_{X+f}(A_n,\epsilon)$, where $A_n$ are finite sets. By the continuity of the volume (see~\cite[Lemma~4.1]{PeresSousi}) $V_{X+f}(A_n,\epsilon)$ is a random variable for each $n$.
}
\end{remark}

\begin{proof}[{\bf Proof of Proposition~\ref{pro:general}}]
The statement of the proposition follows directly from Lemma~\ref{lem:volconc} and Claim~\ref{cl:minkvolgeneral}.
\end{proof}

\subsection{Another proof of Theorem~\ref{thm:constancydimgraph}}\label{sec:levy}

In this section we give an alternative proof of Theorem~\ref{thm:constancydimgraph} that relies on L{\'e}vy's construction of Brownian motion. The only properties of Minkowski dimension that are used in this proof are stability under finite unions and under adding linear functions.

\begin{proposition}
	Let $f:[0,1] \to \R^d$ be a bounded measurable function, and $\mu \in \R^d$. Define $g:[0,1] \to \R^d$ by
	$$
	g(t) = f(t) + \mu t.
	$$
	Then, for every subset $A$ of $[0,1]$, we have
	$$
	\underline{\dim}_M G_A(f) = \underline{\dim}_M G_A(g) \quad \text{and} \quad  \overline{\dim}_M G_A(f) = \overline{\dim}_M G_A(g).
	$$
\end{proposition}

\begin{proof}
	For $\epsilon \in (0, \infty)$ and $k \in \N$ define
	$$
	\C_\epsilon(k) = [(k-1) \epsilon, k \epsilon] \times\{\text{some cube of edge length $\epsilon$ in $\R^d$}\} \ \text{ and } \
	\C_\epsilon = \bigcup_{k \in \N} \C_\epsilon(k).
	$$
	Write $N = \lceil \|\mu\|_\infty\rceil$, and consider a covering of $G_A(f)$ by cubes of $\C_\epsilon$. Consider the cubes of the covering that are in $C_\epsilon(k)$, and thus form a covering of $G_{A\cap[(k-1) \epsilon, k \epsilon]}(f)$. Clearly, shifting them by the vector $(0,\mu_1(k-1)\epsilon,\ldots,\mu_d(k-1)\epsilon)$ produces a covering of
	$$
	G_{A\cap[(k-1) \epsilon, k \epsilon]}(f+ \mu (k-1) \epsilon).
	$$
	But within a time interval of length $\epsilon$, the drift cannot move $f(t)$ by more than $N \epsilon$ in any given direction. Therefore, $	G_{A \cap[(k-1)\epsilon, k \epsilon]}(g)$ may be covered with $N^d$ as many cubes of $\C_\epsilon(k)$ as are required to cover $G_{A \cap[(k-1)\epsilon, k \epsilon]}(f)$.
	It follows that the covering number of $G_A(g)$ with elements of $\C_\epsilon$ is at most $N^d$ times that of $G_A(f)$. Therefore,
	$$
	\underline{\dim}_M G_A(g) \leq \underline{\dim}_M G_A(f) \quad \text{and} \quad \overline{\dim}_M G_A(g) \leq \overline{\dim}_M G_A(f).
	$$
	Since $f(t) = g(t)- \mu t$, the same argument shows that
	$$
		\underline{\dim}_M G_A(f) \leq \underline{\dim}_M G_A(g) \quad \text{and} \quad \overline{\dim}_M G_A(f) \leq \overline{\dim}_M G_A(g),
	$$
	and completes the proof.
\end{proof}

The stability of Minkowski dimension under finite unions yields the following corollary.

\begin{corollary}\label{cor:dimaffinedrift}
	Let $f:[0,1] \to \R^d$ be a bounded measurable function, and $h: [0,1] \to \R^d$ be piecewise affine. Put $g = f+h$. Then, for every subset $A$ of $[0,1]$, we have
	$$
	\underline{\dim}_M G_A(f) = \underline{\dim}_M G_A(g) \quad \text{and} \quad  \overline{\dim}_M G_A(f) = \overline{\dim}_M G_A(g).
	$$
\end{corollary}

\begin{proof}[{\bf Proof of Theorem \ref{thm:constancydimgraph}}]

	We only prove the result for the lower Minkowski dimension. The proof for the upper Minkowski dimension is identical.
	
	Consider L\'evy's construction of Brownian motion as
	$$
	B = \lim_{n \to \infty} Y_n = \lim_{n \to \infty} \sum_{k=1}^n X_k,
	$$
	where $(X_k, k \in \N)$ is an independent sequence of continuous piecewise affine random paths on $[0,1]$, and the convergence is uniform on $[0,1]$.
	
	For $n \in \N$, put $Z_n = B- Y_{n-1}$. Since $	B+f = Z_n + f + Y_{n-1}$	and $Y_{n-1}$ is piecewise affine, Corollary~\ref{cor:dimaffinedrift} implies that $\underline{\dim}_M G_A (B+f) = \underline{\dim}_M G_A (Z_n + f)$.
	In particular, for any $a>0$,
	$$
	\{ \underline{\dim}_M G_A (B+f) \leq a \} = \{\underline{\dim}_M G_A (Z_n + f) \leq a \} \in \sigma(X_k, k \geq n).
	$$
	Since this is true for every $n$, it follows that
	$$
	\Lambda_a=\{ \underline{\dim}_M G_A (B+f) \leq a \} \in \T =\bigcap_{n} \sigma(X_k,k\geq n)
	$$
	Therefore by Kolmogorov's 0-1 law $\pr{\Lambda_a} \in \{0,1\}$. It follows that the Minkowski dimension of $G_A(B+f)$ is almost surely constant.
\end{proof}

\begin{remark}\rm{
An alternative proof of Theorem~\ref{thm:constancydimimage} can be obtained by combining the above proof with Howroyd's projection theorem~\cite[Theorem~14]{Howroyd}.
}
\end{remark}

\section{Dimension of the image of $B+f$}\label{sec:image}

In this section we prove Theorems~\ref{thm:minkba} and~\ref{thm:dimimagelinearBM}.
We first recall Theorem~1.1 from Peres and Sousi~\cite{PeresSousi}, since it is going to be used to prove that the Minkowski dimension of the image and the graph of $B+f$ is larger than that of $B$.

\begin{theorem}[\cite{PeresSousi}]\label{thm:wienersausage}
Let $(B(s))_{s\geq 0}$ be a standard Brownian motion in $d\geq 1$ dimensions and let $(D_s)_{s \geq 0}$ be open sets in $\R^d$. For each $s$, let $r_s>0$ be such that $\vol(\B(0,r_s)) = \vol(D_s)$. Then for all $t$ we have that
\[
\E{\vol\left(\cup_{s \leq t}\left( B(s) + D_s  \right)\right)}\geq \E{\vol\left(\cup_{s \leq t}\B(B(s),r_s)\right)}.
\]
\end{theorem}

\begin{definition}\rm{
Let $G\subseteq \R^d$. We call a collection of balls $(\B(x_i,\epsilon))_i$ an $\epsilon$-packing of $G$ if $x_i \in G$ for all $i$ and the balls are pairwise disjoint.
}
\end{definition}

Given an $\epsilon$-packing of $f(A)$ by $P$ balls with centers $(f(t_i))$ we want to construct an $\epsilon$-packing of $(B+f)(A)$. The balls of radius $\epsilon$ centered at $(B+f)(t_i)$ might not all be disjoint; the following lemma controls the number of collisions.

\begin{lemma}\label{lem:technicalestimate2}
	Let $(B_t)$ be a standard Brownian motion in $d$ dimensions, $f:[0,1] \to \R^d$ a bounded measurable function and $A$ a subset of $[0,1]$.
Then there exists a positive constant $c$ such that for all $\epsilon >0$, if $(\B(f(t_i),\epsilon))_{i\leq P_\epsilon}$ is an $\epsilon$-packing of $f(A)$, then
	$$
	\max_{i\leq P_\epsilon} \E{N_i} \leq c \log (1/\epsilon)^{d+1},
	$$
	where $N_i = \#\{j: |(B+f)(t_i) -(B+f)(t_j)| < 2\epsilon\}$, for $i \in \{1, \dots, P_\epsilon\}$.
\end{lemma}

\begin{proof}[{\bf Proof}]

Let $\epsilon>0$ and $(\B(f(t_i),\epsilon))_{i\leq P_\epsilon}$ an $\epsilon$-packing of $f(A)$.
We fix $i \in \{1,\ldots, P_\epsilon\}$. For every $k\in \N$ we define the sets
\begin{align*}
		S(k) 		& = \{ j : |f(t_i) - f(t_j)| \in [2^k \epsilon, 2^{k+1}\epsilon)\},\\
		S_1(k)		& = \left\{ j \in S(k) : |t_i - t_j| \geq \left( \frac{2^k\epsilon}{\log (1/\epsilon)}\right)^2\right\},\\
		S_2(k)	& = S(k) \setminus S_1(k).
	\end{align*}
Since $f$ is bounded, it follows that $S(k) = \emptyset$ whenever $k \geq c_1 \log(1/\epsilon)$, for a positive constant $c_1$.
Furthermore, if $j \in S(k)$, then since $(f(t_j))_j$ is an $\epsilon$-packing of $f(A)$, the balls 
$\{\B(f(t_j),\epsilon)\}_j$ are disjoint and for all such $j$ the ball $\B(f(t_j),\epsilon)$ is contained in $\B(f(t_i), (2^{k+1}+1)\epsilon)$.
Therefore
	\begin{equation}\label{eq:boundnumber}
		|S(k)| \leq \frac{\vol(\B(0, (2^{k+1}+1)\epsilon))}{\vol (\B(0, \epsilon))} \leq (2^{k+1} + 1)^d.
	\end{equation}
If we write $p(j) = \pr{|(B+f)(t_i) - (B+f)(t_j)| < 2\epsilon}$, then by the definition of $N_i$ we have
	\begin{align}\label{eq:expecni}
	\E{N_i} = \sum_{j=1}^{P_\epsilon} p(j) = \sum_{k=1}^{c_1\log(1/\epsilon)} \sum_{j \in S_1(k)} p(j) + \sum_{k=1}^{c_1\log(1/\epsilon)} \sum_{j \in S_2(k)} p(j).
	\end{align}
If $j \in S_1(k)$, then for a positive constant $c_2$ we have \begin{equation}\label{eq:estimateproba1}
	\begin{aligned}
		p(j) & = \frac{1}{(2\pi |t_i -t_j|)^{d/2}} \int_{\B(f(t_i)-f(t_j), 2 \epsilon)} \exp\left\{-\frac{|x|^2}{2|t_i- t_j|} \right\} dx\\
		& \leq \frac{\log(1/\epsilon)^d}{2^{dk} \epsilon^d (2\pi)^{d/2}} \vol(\B(0, 2 \epsilon))= c_2 \frac{\log(1/\epsilon)^d}{2^{dk} }.
	\end{aligned}
	\end{equation}
If $j\in S_2(k)$, then for a positive constant $c_3$ we have by the Gaussian tail estimate if $k\geq 2$
\begin{equation}\label{eq:estimateproba2}
		\begin{aligned}
			p(j)  \leq \pr{|B(t_i)-B(t_j)| > |f(t_i)-f(t_j)| - 2\epsilon}
				&\leq \pr{|B(t_i)-B(t_j)| > (2^k-2) \epsilon} \\
				 &\leq 2\exp\left(-c_3(\log(1/\epsilon))^2\right).
		\end{aligned}
	\end{equation}
Plugging the estimates~\eqref{eq:estimateproba1} and~\eqref{eq:estimateproba2} in~\eqref{eq:expecni} and using~\eqref{eq:boundnumber} concludes the proof of the lemma.
\end{proof}

\begin{proof}[{\bf Proof of Theorem~\ref{thm:minkba}}]

From Proposition~\ref{pro:general} we infer that a.s.
\begin{align}\label{eq:dimlogvol}
\updim (B+f)(A) = \limsup_{\epsilon \to 0} \frac{\log \E{V_{B+f}(A,\epsilon)}}{\log \frac{1}{\epsilon}}
\\
\label{eq:dimlogvol2} \lowdim (B+f)(A) = \liminf_{\epsilon \to 0} \frac{\log \E{V_{B+f}(A,\epsilon)}}{\log \frac{1}{\epsilon}}.
\end{align}
Let $D_s = \B(f(s),\epsilon)$ if $s\in A$ and $D_s = \emptyset$ if $s\notin A$. Then applying Theorem~\ref{thm:wienersausage} we get
\begin{align}\label{eq:volumesausage}
\E{V_{B+f}(A,\epsilon)} \geq \E{V_{B}(A,\epsilon)}.
\end{align}
From~\eqref{eq:dimlogvol},~\eqref{eq:dimlogvol2} and~\eqref{eq:volumesausage} we deduce the a.s.\ inequalities
\[
\updim (B+f)(A)  \geq \updim B(A) \ \text{ and } \ \lowdim (B+f)(A)  \geq \lowdim B(A).
\]
It now remains to show that almost surely
\begin{align}\label{eq:goalmaximum}
\updim (B+f)(A) \geq \updim f(A) \ \text{ and } \ \lowdim (B+f)(A) \geq \lowdim f(A).
\end{align}

First we note that in the definition of upper and lower Minkowski dimension it suffices to take $\epsilon$ which is tending to $0$ along powers of $2$.

We fix $k$ and consider a $2^{-k}$-packing of $f(A)$ with $P_k(f)$ balls. Let the centers of the balls be $f(t_i)$ with $t_i \in A$ for $i\in \{1,\ldots, P_k(f)\}$.

For every $i\in \{1,\ldots, P_k(f)\}$ define
	$$
	N_i = \#\{ j \neq i : |(B+f)(t_i) - (B+f)(t_j)| < 2^{1-k}\} \ \text{ and } \ G= \sum_{i=1}^{P_k(f)} \1(N_i  < k^2 \E{ N_i}).
	$$
	We call a point $t_i$ \emph{good} if $N_i < k^2 \E{ N_i}$ and \emph{bad} otherwise. Thus $G$ counts the number of good points.
By applying Markov's inequality twice we get
\begin{align}\label{eq:goodpoints}
\pr{G \leq \frac{P_k(f)}{2}}  = \pr{\sum_{i=1}^{P_k(f)} \1(N_i \geq k^2\E{N_i}) \geq \frac{P_k(f)}{2}} \leq \frac{2}{k^2}.
\end{align}
We now want to get a $2^{-k}$-packing of $(B+f)(A)$ from the packing of $f(A)$.
We do this by recursively picking good points $t_i$ and removing the $N_i$ balls $\B((B+f)(t_j),2^{-k})$ that intersect $\B((B+f)(t_i),2^{-k})$. This leaves us with a $2^{-k}$-packing of $(B+f)(A)$ with $\Psi_k$ balls, which on the event $\{G \geq P_k(f)/2\}$ satisfies
\[
\Psi_k \geq \frac{P_k(f)}{2(1+k^2 \max_i \E{N_i})} \geq \frac{P_k(f)}{2(1+ c k^2 \log(2^k)^{d+1})},
\]
where the last inequality follows from Lemma~\ref{lem:technicalestimate2}.
From~\eqref{eq:goodpoints} we now deduce that
\[
\pr{\Psi_k < \frac{P_k(f)}{2(1+ c k^2 \log(2^k)^{d+1})}} \leq \frac{2}{k^2},
\]
and hence by Borel Cantelli we conclude that a.s.\ eventually in $k$
\[
\Psi_k \geq \frac{P_k(f)}{2(1+ c k^2 \log(2^k)^{d+1})}.
\]
Taking $\log$ of both side, dividing by $\log(2^k)$ and taking $\limsup$ and $\liminf$ as $k \to \infty$ finishes the proof.
\end{proof}

\begin{remark}\rm{
We note that $\dim_M(B+f)(A)$ can be much larger than $\max\{ \dim_M B(A), \dim_M f(A)\}$. We recall an example given in~\cite[Example~5.4]{PS10}.
Let $d = 3$ and let $f(t)= (f_1(t),0,0)$, where $f_1$ is a fractional Brownian motion  independent of $B$ of Hurst index $\alpha$. Then
$\dim_M f[0,1] =1$ a.s. For $\alpha$ small we have that almost surely
$\dim_H (B+f)[0,1] = 3-2\alpha$, which is a special case of~\cite[Theorem~1]{Cuzick}. Since $\dim_M (B+f)[0,1] \geq \dim_H (B+f)[0,1]$, we get $\dim_M (B+f)(A) \geq 3-2\alpha$.
}
\end{remark}

We will now prove Corollary~\ref{thm:dimimagelinearBM}.
We start with a standard result about H\"older continuous functions and we include its proof here for the sake of completeness.

\begin{claim}\label{cl:holder}
Let $g: \R_+ \to \R$ be a $\gamma$-H\"older continuous function and $A_\beta = \{ n^{-\beta} : n \in \N\} \cup \{0\}$ for $\beta \in (0, \infty)$. Then
	$$
	\overline{\dim}_M g (A_\beta) \leq \frac{1}{1 + \gamma \beta}.
	$$
\end{claim}

\begin{proof}
Without loss of generality we can assume that $g(0)=0$. Let $L$ be the H\"older constant of $g$. For $k \in \N$ and $n \geq k$, we have
	$$
	|g(n^{-\beta})| \leq L n^{-\gamma \beta} \leq L k^{-\gamma \beta}.
	$$
	Fix $\epsilon > 0$. The set $\{g\left(n^{-\beta}\right): n > k\}\cup \{0\}$ may be covered with $\lceil 2Lk^{-\gamma \beta} /\epsilon\rceil$ closed balls of diameter $\epsilon$. The set $\{g\left(n^{-\beta}\right) : n \leq k\}$ may be covered with $k$ such closed balls. Therefore, the covering number satisfies
	$$
	N(\epsilon) \leq \left\lceil \frac{2L k^{-\gamma \beta}}{\epsilon}\right\rceil + k.
	$$
	Taking $k$ of the order $\epsilon^{-1/(\gamma \beta +1)}$ shows that
	$$
	N(\epsilon) \leq c \epsilon^{-1/(1+\gamma \beta)},
	$$
	and hence the result follows immediately.
\end{proof}

\begin{proof}[{\bf Proof of Corollary~\ref{thm:dimimagelinearBM}}]

By Theorem~\ref{thm:minkba} it suffices to prove the inequalities for $f=0$. The case $d\geq 2$ follows from~\cite[Lemma~2.3(a)]{PerkinsTaylor}. The case $d=1$ can then be inferred by projecting a planar Brownian motion on a line in a random direction and applying~\cite[Theorem~3]{FalconerHowroyd}. However, we give a self-contained proof below.

We set $\alpha = \updim A$ and $\beta = \lowdim A$.

Again we take $\epsilon$ tending to $0$ along powers of $2$. Let $\delta_k = 2^{-2k/(\alpha+1)}$ and consider a $\delta_k$-packing of $A$ with $P_{\delta_k}(A)$ balls with centers $(t_i)_{i\leq P_{\delta_k}(A)}$.
		We call a point $t_i$ \emph{good} if
		$$
		N_i = \#\{ j \neq i : |B(t_i) - B(t_j)| < 2^{1-k}\} < k^2 \E{N_i},
		$$
		and \emph{bad} otherwise.
Let $G=G_k$ denote the number of good points. Then by Markov's inequality as in~\eqref{eq:goodpoints} we get that
\[
\pr{G \leq \frac{P_{\delta_k}(A)}{2}} \leq \frac{2}{k^2}.
\]

We now want to get a $2^{-k}$-packing of $B(A)$ from the $\delta_k$-packing of $A$.
We do this by recursively picking good points $t_i$ and removing the
 $N_i$ balls $\B(B(t_j),2^{-k})$ that intersect $\B(B(t_i), 2^{-k})$. This yields a $2^{-k}$-packing of $B(A)$ with $\Psi_k$ balls, which on the event $\{G\geq P_{\delta_k}(A)/2\}$ satisfies
		\begin{equation}\label{eq:packingnumbergoodpoints}
		\Psi_k \geq \frac{P_{\delta_k}(A)}{2(1+ k^2 \max_i \E{N_i})}.
		\end{equation}
Since the points $(t_i)$ are a $\delta_k$-packing of the set $A$, it follows that $|t_{i+\ell} - t_i| \geq \delta_k \ell$, and hence we  can bound
\begin{align*}
\E{N_i} = \sum_{\substack{j=1 \\ j \neq i}}^{P_{\delta_k}(A)} \pr{|(B+f)(t_i) - (B+f)(t_j)| < 2^{1-k}} \leq \sum_{\substack{\ell = -i+1 \\ \ell \neq 0}}^{P_{\delta_k}(A)} \frac{c 2^{-k}}{\sqrt{|\ell | \delta_k}} \leq c'  \sqrt{P_{\delta_k}(A)}.
\end{align*}
Substituting this in~\eqref{eq:packingnumbergoodpoints} we get
\[
\pr{\Psi_k < \frac{P_{\delta_k}(A)}{2(1+ c' k^2\sqrt{P_{\delta_k}(A)}) }} \leq \frac{2}{k^2},
\]
and hence by Borel Cantelli again we get that a.s.\ eventually in $k$
\begin{align}\label{eq:psi}
\Psi_k \geq \frac{P_{\delta_k}(A)}{2(1+ c' k^2 \sqrt{P_{\delta_k}(A)} )}.
\end{align}
Taking $\log$ of both sides, dividing by $\log(2^k)$ and taking $\limsup$ and $\liminf$ as $k\to \infty$ concludes the proof of the first part of the corollary.

%

For the second part, let $\alpha>0$ and set $\beta = 1/\alpha -1$. It is easy to check that the set $A_\beta$ from Claim~\ref{cl:holder} satisfies $\dim_M A_\beta=(1+\beta)^{-1}= \alpha$.
Now pick $\gamma < 1/2$. Since almost all Brownian paths are $\gamma$-H\"older continuous, Claim~\ref{cl:holder} guarantees that
	$$
	\overline{\dim}_M B(A_\beta) \leq \frac{1}{1+\gamma \beta} \to \frac{1}{1 + \beta/2} = \frac{2 \alpha}{\alpha + 1} \text{ as } \gamma \to 1/2.
	$$
\end{proof}

\section{Dimension of the graph}\label{sec:graph}

In this section we give the proofs of the theorems stated in the Introduction concerning the Minkowski dimension of the graph of $B+f$.

We start by proving a lemma analogous to Lemma~\ref{lem:technicalestimate2} in the case of the graph.

\begin{lemma}\label{lem:technicalestimate3}
	Let $(B_t)$ be a standard Brownian motion in $d$ dimensions, $f:[0,1] \to \R^d$ a bounded measurable function, and $A$ a subset of $[0,1]$.
Then there exists a positive constant $c$ such that for all $\epsilon>0$, if $(\B((t_i,f(t_i)),2 \epsilon))_{i\leq P_\epsilon}$ is a $4\epsilon$-packing of $G_A(f)$, then
\[
\max_{i\leq P_\epsilon} \E{N_i} \leq c \log (1/\epsilon)^{d+1},
\]
where $N_i = \#\{j: |(t_i,(B+f)(t_i)) -(t_j,(B+f)(t_j))| < 2\epsilon\}$, for
	 $i \in \{1, \dots, P_\epsilon\}$.
\end{lemma}

\begin{proof}[{\bf Proof}]

Let $\epsilon>0$ and $(\B((t_i,f(t_i)),4\epsilon))_{i\leq P_\epsilon}$ a $4\epsilon$-packing of $G_A(f)$. We fix $i\in \{1,\ldots,P_\epsilon\}$. For every $k\in \N$ we define the sets
\begin{align*}
		S(k) 		& = \{ j : |t_i -t_j| <2 \epsilon \ \text{ and } \ |f(t_i) - f(t_j)| \in [2^k \epsilon, 2^{k+1}\epsilon)\},\\
		S_1(k)		& = \left\{ j \in S(k) : |t_i - t_j| \geq \left( \frac{2^k\epsilon}{\log (1/\epsilon)}\right)^2\right\},\\
		S_2(k)	& = S(k) \setminus S_1(k).
	\end{align*}
Again since $f$ is bounded, $S(k)=\emptyset$ when $k \geq c_1\log(1/\epsilon)$. Furthermore, if $j\in S(k)$, then since $(t_j,f(t_j))_j$ is a $4\epsilon$-packing of $G_A(f)$, the balls $\{\B(f(t_j),\epsilon)\}_j$ are disjoint and for all such $j$ the ball $\B(f(t_j),\epsilon)$ is contained in $\B(f(t_i), (2^{k+1}+1)\epsilon)$. Hence
\begin{equation}\label{eq:boundnumbergraph}
		|S(k)| \leq \frac{\vol(\B(0, (2^{k+1}+1)\epsilon))}{\vol (\B(0, \epsilon))} \leq (2^{k+1} +1)^d.
	\end{equation}
	We now set
	\[
	q(j) = \pr{|( t_i, (B+f)(t_i)) - ( t_j, (B+f)(t_j))| < 2\epsilon}
		\leq \pr{| (B+f)(t_i) - (B+f)(t_j) | < 2 \epsilon}.
	\]
Proceeding as for the estimate for $p(j)$ in Lemma~\ref{lem:technicalestimate2} gives that if $j \in S_1(k)$, then for a positive constant $c_2$
	$$
	q(j) \leq c_2 2^{-dk} \log(1/\epsilon)^d.
	$$
Similarly to the proof of Lemma~\ref{lem:technicalestimate2}, if $j \in S_2(k)$ for $k\geq 2$, we get for a positive constant $c_3$
\[
q(j) \leq 2\exp\left(-c_3(\log(1/\epsilon))^2\right).
\]
Plugging these two estimates above in the expression for $\E{N_i}$ and using~\eqref{eq:boundnumbergraph} we deduce
\[
\E{N_i} \leq c \log(1/\epsilon)^{d+1},
\]
where $c$ is a positive constant
and this concludes the proof of the lemma.
\end{proof}

\begin{proof}[{\bf Proof of Theorem~\ref{thm:dimgraphBMbrownian}}]

Let $\epsilon>0$ and
\[
\C_{\epsilon,d+1} = \{ [(\ell_1-1)\epsilon,\ell_1\epsilon] \times \ldots \times [(\ell_{d+1} -1)\epsilon,\ell_{d+1} \epsilon]:\ \ \ell_1,\ldots,\ell_{d+1} \in \Z\}.
\]
By~\cite[Definition~3.1]{Falconer} the Minkowski dimension  of $G_A(B+f)$ is determined by counting the boxes in $\C_{\epsilon,d+1}$ that intersect $G_A(B+f)$.
We take $\epsilon$ tending to $0$ along powers of $2$, i.e.\ take $\epsilon=2^{-n}$. Let the minimal number of boxes in the covering be $N(\epsilon) = N_n$. Setting $I_k = [(k-1)2^{-n},k2^{-n}]$ and
\[
A_{k,n} =\# \{\text{boxes in } \C_{\epsilon,d} \text{ intersecting } (B+f)(I_k \cap A) \}
\]
we have $N_n= \sum_{k=1}^{2^n} A_{k,n}$. Just like in the proof of Claim~\ref{cl:minkvolgeneral} we get that there exist positive constants $c_1$ and $c_2$ such that
\begin{align*}
 c_1 2^{nd} \sum_{k=1}^{2^n} V_{B+f}(I_k\cap A,2^{-n})  &\leq N_n \leq c_2 2^{nd} \sum_{k=1}^{2^n} V_{B+f}(I_k\cap A,2^{-n})   \ \text{ and } \\
 c_1 2^{(d+1)n}V_{G(B+f)}(A,2^{-n})  & \leq N_n \leq c_2 2^{(d+1)n}V_{G(B+f)}(A,2^{-n}),
\end{align*}
where $G(B+f)$ stands for the process $(t,B(t) + f(t))_t$.

Since the process $(s,B(s))$ is c{\`a}dl{\`a}g and has independent and stationary increments and $(s,f(s))$ is c{\`a}dl{\`a}g, Lemma~\ref{lem:volconc} together with the above inequalities give that a.s.
\[
\updim G_A(B+f)= \limsup_{n\to \infty} \frac{\log \left( \sum_{k=1}^{2^n} \E{V_{B+f}(I_k\cap A,2^{-n})}  \right)}{\log(2^n)} + d.
\]
Fix $k \in \{1,\ldots,2^n\}$. For $s\in I_k \cap A$ define $D_s = \B(f(s),2^{-n})$ and for $s\notin I_k\cap A$ let $D_s = \emptyset$. Then Theorem~\ref{thm:wienersausage} gives
\[
\E{V_{B+f}(I_k\cap A,2^{-n})}  \geq \E{V_{B}(I_k\cap A,2^{-n})  },
\]
and hence it follows that a.s.
\[
\updim G_A(B+f) \geq \updim G_A(B).
\]
The inequality for the lower Minkowski dimension of $G_A(B+f)$ follows in exactly the same way.

It now remains to show that a.s.
\[
\updim G_A(B+f) \geq \updim G_A(f)
\]
and similarly for lower Minkowski.
The proof of that follows in the same way as the proof of Theorem~\ref{thm:minkba}.
We point out the differences. We call a point $(t_i,f(t_i))$ good if $N_i < k^2 \E{N_i}$, where $N_i$ is as defined in the statement of Lemma~\ref{lem:technicalestimate3} for $\epsilon = 2^{-k}$.
Then we proceed in exactly the same way as in Theorem~\ref{thm:minkba} and on the event that the number of good points is at least $P_k(f)/2$ we get a $2^{-k}$-packing of $G_A(f)$ with at least
\[
\frac{P_k(f)}{2(1+k^2 \max_i \E{N_i})}
\]
balls of radius $2^{-k}$. Using Lemma~\ref{lem:technicalestimate3} and the Borel Cantelli lemma as in Theorem~\ref{thm:minkba} concludes the proof.
\end{proof}

The rest of this section is devoted to the proof of Theorem~\ref{thm:graphlinearBM}.

\begin{proposition}\label{pro:minkdimsumfg}
	Let $f$ and $g:[0,1] \to \R$ be two continuous functions. Assume that $\dim_M G_{[0,1]}(f)$ exists. Then,
	\begin{align*}
	\underline{\dim}_M G_{[0,1]} (f+g) &\leq \max\{{\dim}_M G_{[0,1]}(f), \underline{\dim}_M G_{[0,1]}(g)\},\\
	\overline{\dim}_M G_{[0,1]} (f+g) &\leq \max\{{\dim}_M G_{[0,1]}(f), \overline{\dim}_M G_{[0,1]}(g)\}.
	\end{align*}
	Furthermore, in both cases, when the dimensions on the right hand side are different, we even have equality.
\end{proposition}

\begin{proof}[{\bf Proof}]

	We shall only prove the inequality for the lower Minkowski dimension. The other case is proved similarly.
	
Set $\alpha = \dim_M G_{[0,1]}(f)$ and $\beta = \lowdim G_{[0,1]}(g)$ and consider the collection of squares
	$$
	\C_\epsilon = \{ [(k-1)\epsilon, k \epsilon] \times [(\ell-1) \epsilon, \ell \epsilon]: k, \ell \in \Z\}.
	$$
	Let $h:[0,1]\to \R$ be a continuous function.
	A covering of $G_{A}(h)$ is given by taking all the elements of $\C_\epsilon$ that intersect $G_{[0,1]}(h)$; and that many are needed. Let $S_\epsilon(h)$ be the number of these squares. Take $\epsilon=2^{-n}$ and set $I_k= [(k-1)\epsilon,k\epsilon]$ and $\Omega_{k,n}(h) = \lceil 2^n(\max_{s\in I_k} h(s) - \min_{s\in I_k} h(s)) \rceil$. Then it is easy to see that
\begin{align}\label{eq:numberofsquares}
\sum_{k=1}^{2^n} \Omega_{k,n}(h)\leq S_{2^{-n}}(h) \leq  2\sum_{k=1}^{2^n}  \Omega_{k,n}.
\end{align}		
It is straightforward to check that
\begin{align}\label{eq:twosides}
\Omega_{k,n}(g) -\Omega_{k,n}(f)  \leq \Omega_{k,n}(f+g) \leq \Omega_{k,n}(f) + \Omega_{k,n}(g).
\end{align}
Let $(\epsilon_n)$ be a subsequence of $(2^{-n})$ along which $\log S_{\epsilon_n}(g)/\log(1/\epsilon_n) \to \beta$ as $n\to \infty$. Fix $\delta\in (0,\infty)$. Then $S_{\epsilon_n}(g) \leq \epsilon_n^{-\beta -\delta}$ for all $n$ large enough. Since $\dim_M G_{[0,1]}(f)$ exists, it follows that for all $n$ large enough $S_{\epsilon_n}(f) \leq \epsilon_n^{-\alpha - \delta}$. Thus for all $n$ sufficiently large we obtain
\[
S_{\epsilon_n}(f+g) \leq 2S_{\epsilon_n}(f) + 2S_{\epsilon_n}(g) \leq 4\epsilon_n^{-\max\{\alpha,\beta\} - \delta}.
\]
Taking logarithms of both sides, dividing by $\log(1/\epsilon_n)$ and taking the limit as $n\to \infty$ gives that for all $\delta>0$
\[
\lowdim G_{[0,1]}(f+g) \leq \max\{\alpha,\beta\} + \delta,
\]
and hence letting $\delta\to 0$ gives $\lowdim G_{[0,1]}(f+g) \leq \max\{\alpha,\beta\}$.

It only remains to show the final statement of the proposition. Suppose that $\alpha <\beta$. We will show that
\[
\lowdim G_{[0,1]}(f+g) \geq \beta.
\]
The other cases are treated similarly.

Take $\delta>0$ small enough
so that $\beta >\alpha + 2\delta$. Using~\eqref{eq:numberofsquares} and the left hand side inequality of~\eqref{eq:twosides} we deduce that for all $n$ sufficiently large
\[
S_n(f+g) \geq \frac{2^{n(\beta - \delta)} }{2} - 2^{n(\alpha+\delta)},
\]
and hence it easily follows that in this case $\lowdim G_{[0,1]}(f+g) \geq \beta$, which together with the inequality previously shown completes the proof.
\end{proof}

\begin{proof}[{\bf Proof of Theorem \ref{thm:graphlinearBM}}]
The theorem follows directly from Proposition~\ref{pro:minkdimsumfg} and Theorem~\ref{thm:dimgraphBMbrownian}.
\end{proof}

\section{Example}\label{sec:example}

\begin{example}
Let $\phi:\R_+ \to \R$ be a function with period $1$ defined in $[0,1]$ via $\phi(x)=\max\{x,1-x\}$.
For $n\in \N$ we define $\psi_n:[0,1] \to \R_+$ via
\[
\psi_n(x)  = n^{-3/4} \lfloor \sqrt{n} \phi(nx) \rfloor
\]
adjusted to be c{\`a}dl{\`a}g. Let $n_k = 2^{6^k}$ and define $f = \sum_{k=1}^{\infty} \psi_{n_k}$. Since $f$ is the uniform
limit of c{\`a}dl{\`a}g functions, it is also c{\`a}dl{\`a}g.
We will show that
\[
\updim G_{[0,1]}(f) = \frac{5}{3} \ \text{ and } \ \updim G_{[0,1]}(B+f) \geq \frac{7}{4} \text{ a.s.}
\]
\end{example}

\begin{figure}
\begin{center}
\epsfig{file=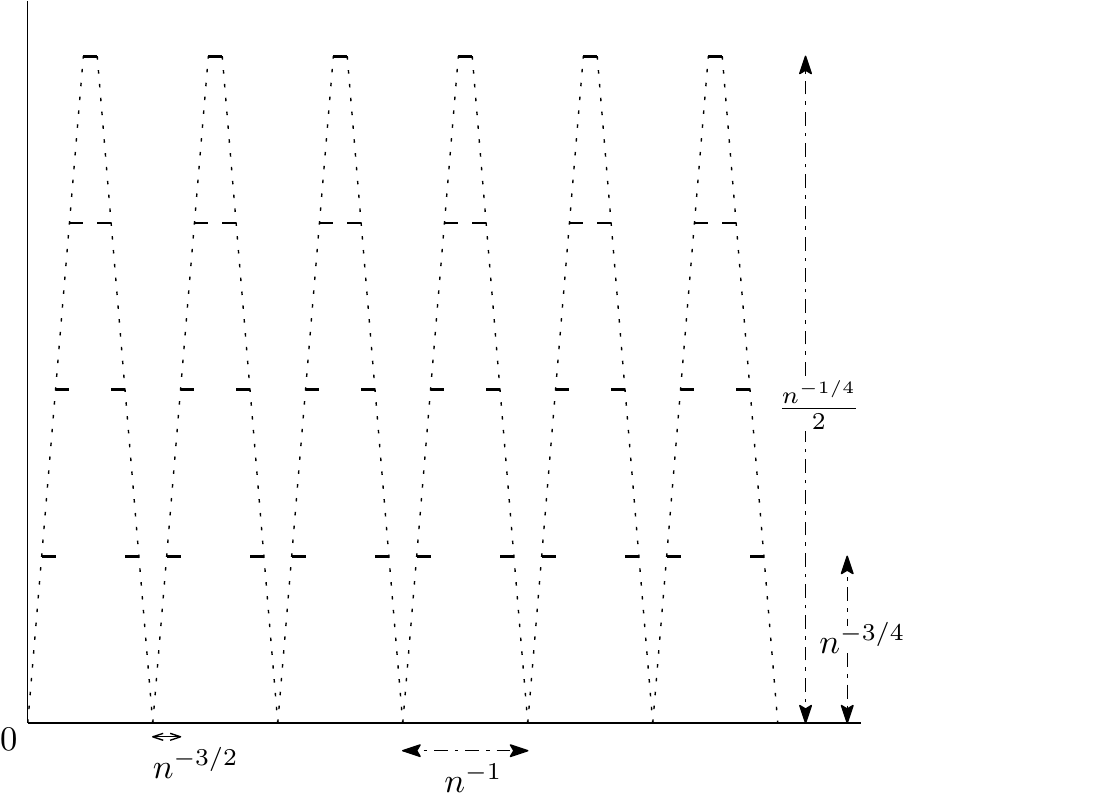, scale=1}
\caption{\label{fig:psin}Graph of $\psi_n$}
\end{center}
\end{figure}

The idea motivating this construction is that the graph of $f$ can be covered efficiently due to the large jumps; $B+f$ interpolates many of these jumps and hence has a larger Minkowski dimension. Moreover, the graph of $B$ can be covered efficiently due to cancellation of the upward and downward movement; adding $f$ to $B$ eliminates much of this cancellation, so that the graph of $B+f$ has Minkowski dimension greater than the graph of $f$.

\begin{claim}
$\updim G_{[0,1]}(f) = \frac{5}{3}$.
\end{claim}

\begin{proof}[{\bf Proof}]
Let $\epsilon = n_\ell^{-3/4}$ and suppose that we want to cover the graph of $f$ with boxes of side length~$\epsilon$. We now argue that the number of boxes needed is up to constants the number of boxes needed to cover the graph of $\psi_{n_\ell}$.

Indeed, since for all $x$
\begin{align}\label{eq:boxes}
\sum_{k=\ell+1}^{\infty} \psi_{n_k}(x) \leq  \frac{1}{2}\sum_{k=\ell+1}^{\infty} \left(2^{6^k}\right)^{-1/4} \leq  c_1\left(2^{6^\ell} \right)^{-3/2},
\end{align}
where $c_1$ is a positive constant, the number of $\epsilon$-boxes needed to cover the graph of $\sum_{k\geq \ell}\psi_{n_k}$ is up to constants the same as the number of $\epsilon$-boxes needed to cover the graph of $\psi_{n_\ell}$.

Note that $n_{k}^{3/2}$ divides $n_{k+1}$ for all $k$.  Also for all $k<\ell$ the function $\psi_{n_{k}}$ is constant on each standard subinterval of length $n_{\ell -1}^{-3/2}$.
Thus on a subinterval of length $n_{\ell-1}^{-3/2}$ adding the functions $\psi_{n_k}$ for $k<\ell$ to the function $\sum_{k=\ell}^{\infty}\psi_{n_k}$ only
shifts the graph of $\sum_{k=\ell}^{\infty}\psi_{n_k}$ by a constant on each standard interval of length $n_{\ell-1}^{-3/2}$. Since we want to cover the graph of $f$ with boxes of side length $n_{\ell}^{-3/4}$ and we can fit an integer number of those in the subinterval of length $n_{\ell-1}^{-3/2}$, it follows that the number of $\epsilon$-boxes needed is the same as the number of $\epsilon$-boxes needed to cover the graph of $\sum_{k=\ell}^{\infty}\psi_{n_k}$.

So it only remains to calculate the number of $\epsilon$-boxes needed to cover the graph of $\psi_{n_\ell}$. By the construction of the function $\psi_{n_\ell}$ it is easy to see that the number of $\epsilon$-boxes needed to cover the graph of $\psi_{n_\ell}$ is of order
$\frac{1}{n_\ell^{-3/4}} \frac{n_\ell^{-1/4}}{n_\ell^{-3/4}} = n_\ell^{5/4}$.
%
%

Therefore the number of boxes of side length
$n_\ell^{-3/4}$ needed to cover the graph of $f$ is of order
$n_{\ell}^{5/4}$. From that it follows that
\[
\updim G(f) \geq \frac{5}{3}.
\]
It remains to show $\updim G(f) \leq \frac{5}{3}$.

Take $\epsilon$ converging to $0$ along powers of $2$. Let $\epsilon = 2^{-r}$ and $k$ be such that
\[
n_k^{-3/2} \leq \epsilon < n_{k-1}^{-3/2}.
\]
We consider two separate cases. To simplify notation we write $n=n_k$.
\begin{itemize}
\item If $\epsilon<n^{-3/4}$, then we need order $ \epsilon^{-1}\sqrt{n}$ boxes to cover the graph of~$f$.
This follows from~\eqref{eq:boxes}, the fact that $\epsilon > n^{-3/2}$ and that $\epsilon$ divides $n_{\ell}^{-3/2}$ for all $\ell<k$.

\item If $\epsilon >n^{-3/4}$, then the number of boxes needed to cover the graph of $\psi_{n}$ is of order $n^{-1/4}/\epsilon^2$. Since the contributions of the other functions in the sum do not matter as discussed above, this is indeed the covering number for the graph of $f$.
\end{itemize}

From the two cases above it follows that $\updim G(f) \leq \frac{5}{3}$ and this concludes the proof.
\end{proof}

\begin{figure}
\begin{center}
\epsfig{file=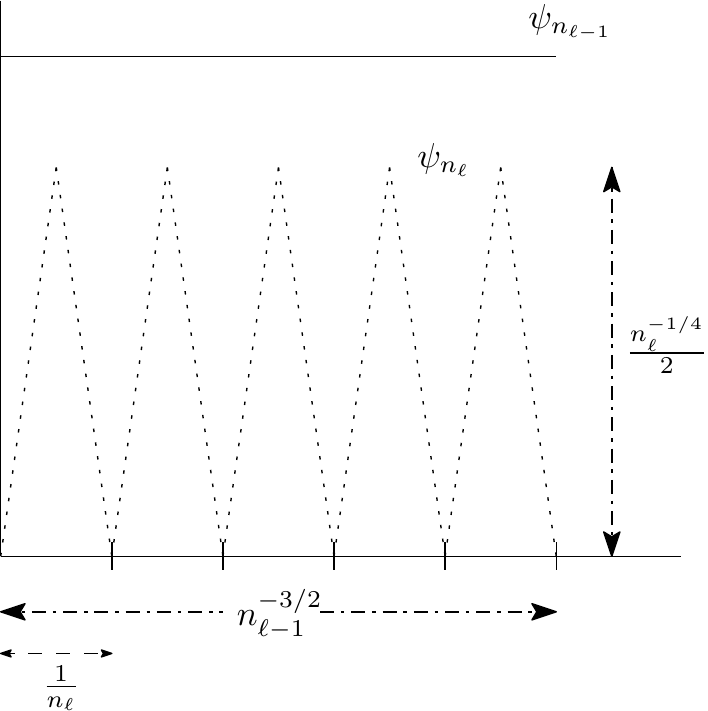, scale=1}
\caption{\label{fig:line}Two scales}
\end{center}
\end{figure}

\begin{claim}
$\updim G(B+f) \geq \frac{7}{4}$ a.s.
\end{claim}

\begin{proof}[{\bf Proof}]

Suppose that we want to cover the graph of $B+f$ with boxes of side length $n_\ell^{-1}$. Then arguing as above it is enough to find the number of $n_\ell^{-1}$-boxes needed to cover the graph of $B+\psi_{n_\ell}$.

First we subdivide the interval $[0,1]$ into subintervals of length $n^{-3/2} + (\log n)^2 n^{-3/2}$. Thus the number of such subintervals we obtain is of order $n^{3/2}/(\log n)^2$. For each such subinterval $I_{j,n}$ we write $s_{j,n} = \inf I_{j,n}$ and we define the events
\[
A_n = \left\{\forall j \  \   B_{s_{j,n}}   - \inf_{t\in I_{j,n}} B_t \leq  \frac{n^{-3/4}(\log n)^2}{2}  \right\}.
\]
Using the Gaussian tail estimate gives
\[
\pr{A_n^c} \leq c_1 \frac{n^{3/2}}{(\log n)^2} e^{-c_2 (\log n)^2},
\]
which is summable. Hence by Borel Cantelli we get that almost surely for all $n$ sufficiently large we get that in none of the intervals $I_{j,n}$ Brownian motion goes down by more than $n^{-3/4}(\log n)^2/2$.

We now look at the first part of these subintervals $\til{I}_{j,n}$ of length $n^{-3/2}$ and we define the event
\[
\til{A}_{j,n} = \left\{ n^{-3/4} \leq \sup_{t \in \til{I}_{j,n}} B_t - B_{\til{s}_{j,n}}     \leq 2n^{-3/4} \right\},
\]
where $\til{s}_{j,n} = \inf \til{I}_{j,n}$.
Then there exists a constant $c\in (0,1)$ so that for all $j$ and $n$
\[
\pr{\til{A}_{j,n}} \geq c.
\]
The events $(\til{A}_{j,n})_j$ are independent by the independence of the increments of Brownian motion. Using the Chernoff bound for Bernoulli random variables we obtain for a positive constant $c_3<1$
\[
\pr{\sum_{j=1}^{n^{3/2}} \1(\til{A}_{j,n}) \geq \frac{cn^{3/2}}{4}} \geq 1- c_3^{n^{3/2}}.
\]
Thus applying Borel Cantelli again we deduce that almost surely for all $n$ large enough at least $cn^{3/2}/4$ of the events $\til{A}_{j,n}$ will happen.

We now take $n$ sufficiently large so that $A_n$ holds and at least $cn^{3/2}/4$ of the events $\til{A}_{j,n}$ occur.

We set $n=n_\ell$ and we consider the subintervals of length $n^{-3/2}$ that correspond to the events $\til{A}_{j,n}$ that occur. In each of these subintervals the function $\psi_{n}$ is constant, and by the definition of the event $\til{A}_{j,n}$, it is easy to see that the number of boxes of side $n^{-1}$ needed to cover the graph of~$B+\psi_{n}$ in this time interval is at least of order $n^{-3/4}/n^{-1} = n^{1/4}$. Next we skip a time interval of length~$n^{-3/2} (\log n)^2$. Since the event $A_n$ holds, during this time interval the Brownian motion did not go down by more than $(\log n)^2 n^{-3/4} /2$. At the same time the function~$f$ increased by $n^{-3/4} (\log n)^2$. So it follows that we need at least of order $n^{1/4} n^{3/2}/(\log n)^2$ boxes of side length $n^{-1}$ to cover the graph of $B+\psi_{n}$. Therefore we deduce that a.s.
\[
\updim G(B+f) \geq \frac{7}{4}.
\]
\end{proof}

%
%
%
%

\begin{remark}
\rm{
A modification of the example yields a continuous function $f$ and a closed set $A$ in $[0,1]$ so that
\[
\updim G_A(f) = \frac{5}{3} \ \text{ and } \ \updim G_A(B+f) \geq \frac{7}{4}.
\]
}
\end{remark}

\section*{Acknowledgement}

We thank Yimin Xiao for providing useful references.

\bibliographystyle{plain}
\bibliography{biblio}

\end{document}